\documentclass[preprint,12pt]{elsarticle}

\usepackage{amssymb,amsmath,amsthm,latexsym}
\usepackage{psfrag}
\usepackage{graphicx}
\usepackage{url}
\usepackage{psfrag}
\usepackage{graphpap}
\usepackage{pstricks}
\usepackage{pst-node}
\usepackage{multicol}
\usepackage{relsize}

\usepackage[left=3cm,right=3cm,top=3cm,bottom=3cm,a4paper]{geometry}

\theoremstyle{plain}
\newtheorem{Thm}{Theorem}
\newtheorem{Lem}[Thm]{Lemma}
\newtheorem{Cor}[Thm]{Corollary}

\newcommand{\cH}{\ensuremath{\mathcal{H}}}

\begin{document}

\begin{frontmatter}

\title{The competition hypergraphs  of doubly partial orders}

\author[label1]{Suh-Ryung KIM}
\author[label2]{Jung Yeun LEE}
\author[label3]{Boram PARK\corref{cor1}}
\author[label4]{Yoshio SANO}

\address[label1]{Department of Mathematics Education,
Seoul National University, Seoul 151-742, Korea}
\address[label2]{National Institute for Mathematical Sciences,
Daejeon 305-390, Korea}
\address[label3]{DIMACS, Rutgers University, Piscataway, NJ 08854, United States}
\address[label4]{National Institute of Informatics,
Tokyo 101-8430, Japan}

\cortext[cor1]{Corresponding author. {\it E-mail address}:
kawa22@snu.ac.kr; borampark22@gmail.com}

\journal{arXiv}

\begin{abstract}
Since Cho and Kim (2005)~\cite{chokim} showed that
the competition graph of a doubly partial order is an interval graph,
it has been actively studied whether or not the same phenomenon occurs
for other variants of competition graphs and
interesting results have been obtained.
Continuing in the same spirit, we study the competition hypergraph,
an interesting variant of the competition graph,
of a doubly partial order.
Though it turns out that the competition hypergraph of
a doubly partial order is not always interval,
we completely characterize the competition hypergraphs
of doubly partial orders which are interval.
\end{abstract}


\end{frontmatter}

\section{Introduction}

Given a digraph $D$, the {\em competition graph} $C(D)$ of $D$
is a graph which has the same vertex set as $D$ and
has an edge between vertices $u$ and $v$
if and only if there exists a common out-neighbor of $u$ and $v$ in $D$.
The notion of the competition graph is due to Cohen~\cite{cohen1}
and has arisen from ecology.
Competition graphs also have applications in coding, radio
transmission, and modeling of
complex economic systems (see \cite{RayRob, Bolyai}).
Since Cohen introduced the notion of the competition graph,
various variations have been defined and studied by many authors
(see \cite{cable, firsti, sc} and the survey articles 
\cite{Kim93, Lundgren89}).

Cohen~\cite{cohen1, cohen2} observed empirically that
most competition graphs of acyclic digraphs representing food webs
are interval graphs.
A graph $G$ is an {\em interval graph}
if we can assign to
each vertex $v$ in $G$ a real interval $J(v) \subseteq \mathbb{R}$
such that
there is an edge between two distinct vertices
$v$ and $w$
if and only if $J(v) \cap J(w) \neq \emptyset$.
Cohen's observation and the continued
preponderance of examples that are interval graphs led to a large
literature devoted to attempts to explain the observation and to
study the properties of competition graphs.
Roberts~\cite{Rob78}
showed that every graph can be made into the competition graph of an
acyclic digraph by adding isolated vertices.  He then asked for
a characterization of acyclic digraphs whose competition graphs
 are interval. The study of acyclic digraphs whose competition graphs are
interval led to several new problems and applications
(see \cite{fi, Fraughnaugh, kimrob, LK, lunmayras, sano}).
As one of the consequences, Cho and Kim~\cite{chokim} found
an interesting class of acyclic digraphs called ``doubly partial orders"
with interval competition graphs.
We denote by $\prec$ the partial order
$\{((x_1,x_2),(y_1,y_2))\mid x_1 < y_1, x_2<y_2\}$ on  $\mathbb{R}^2$.
A digraph $D$ is called a {\it doubly partial order} (a DPO for short)
if there exist a finite subset $V$ of $\mathbb{R}^2$
and a bijection $\phi: V(D) \rightarrow V$
such that $A(D) = \{(x,y) \mid \phi(y) \prec \phi(x), x,y \in V(D) \}$.
The following theorem clarifies the relationship
between interval graphs and the competition graphs of doubly partial orders.

\begin{Thm}[\cite{chokim}]\label{thm;DPOIntKim}
The competition graph of a doubly partial order is an interval graph,
and an interval graph with sufficiently many isolated vertices is
the competition graph  of a doubly partial order.
\end{Thm}

Since then, it has been actively studied whether or not
the same phenomenon occurs for other variants of competition graphs
and interesting results have been obtained.

\begin{Thm}[\cite{SJkim}]
The competition-common enemy graph of a doubly partial order
is an interval graph
unless it contains a $4$-cycle as an induced subgraph.
In addition, an interval graph with sufficiently many isolated vertices
is the competition-common enemy graph of a doubly partial order.
\end{Thm}

\noindent
The above result on competition-common enemy graphs was generalized
by Lu and Wu~\cite{LuWu} and Wu and Lu~\cite{WuLu}.
Most recently, the niche graph,
the $m$-step competition graph,
and the phylogeny graph
of a doubly partial order were studied.

\begin{Thm}[\cite{Niche2009}]
The niche graph of a doubly partial order is an interval graph
unless it contains a triangle.
\end{Thm}

\begin{Thm}[\cite{PLK:mStepDPO}]
For any positive integer $m$,
the $m$-step competition graph of a doubly partial order
is an interval graph, and
an interval graph with sufficiently many isolated vertices is
the $m$-step competition graph of a doubly partial order.
\end{Thm}

\begin{Thm}[\cite{PS:phylo}]
The phylogeny graph of a doubly partial order
is an interval graph.
In addition,
for any interval graph $G$,
there exists an interval graph $\tilde{G}$
such that $\tilde{G}$ contains the graph $G$ as an induced subgraph
and that $\tilde{G}$ is the phylogeny graph of a doubly partial order.
\end{Thm}

Continuing in the same spirit,
we study the competition hypergraph of a doubly partial order.
The notion of a competition hypergraph
which is a variant of a competition graph was introduced
by Sonntag and Teichert \cite{CompHyp}.
The {\it competition hypergraph} $C\cH(D)$ of
a digraph $D$ is a hypergraph without loops and multiple hyperedges
such that the vertex set is the same as 
the vertex set of $D$ and $e \subset V(D)$
is a hyperedge if and only if
$e$ contains at least two vertices and
$e$ coincides with the in-neighborhood  of some vertex $v$
in $D$.
As we study the competition hypergraphs of digraphs,
we assume that all hypergraphs considered in this paper
have no loops and no multiple hyperedges.
The notion of a competition hypergraph is considered
as one of the important variants of competition graphs and
significant results on this topic are being obtained
(see \cite{PS:Hyper, CompHyp, Son2, Son3, Son4}).
In this paper,
we classify doubly partial orders
whose competition hypergraphs are interval.

\section{Main Results}

Throughout this section, we follow 
the terminology for hypergraphs given in \cite{voloshin}.
We say that two vertices $u$ and $v$ are \textit{adjacent}
in a hypergraph $\cH$
if there is a hyperedge $e$ in $\cH$ such that $\{u,v\} \subset e$.
For a positive integer $r$,
a hypergraph $\cH$ is called {\it $r$-uniform}
if each hyperedge of the hypergraph $\cH$ has the same size $r$.
Obviously, $2$-uniform hypergraphs are graphs.
A sequence $v_0 v_1 \cdots v_k$ of distinct vertices of
a hypergraph $\cH$
is called a \textit{path}
if there exist $k$ distinct hyperedges $e_1, e_2, \ldots, e_k$
such that $e_i$ contains $\{v_{i-1}, v_i\} $ for each $1 \leq i \leq k$.
A sequence $v_0 v_1 \cdots v_k$ of distinct vertices of a hypergraph $\cH$
is called a \textit{cycle} if there exist $k+1$ distinct hyperedges
$e_1, e_2, \ldots, e_k,e_{k+1}$
such that $e_i$ contains $\{v_{i-1}, v_i\} $ for each
$1 \le i \le k$ and $e_{k+1}$ contains $\{v_0,v_k\}$.
A \textit{subhypergraph} of a hypergraph $\cH$
is a hypergraph $\cH'$ such that $V(\cH')\subseteq V(\cH)$ and
$E(\cH') = \{ e\cap V(\cH')
\mid e\in E(\cH), |e \cap V(\cH')| \geq 2 \}$.
For a vertex $v$ in a digraph $D$,
we denote by $N_D^-(v)$ the in-neighborhood of $v$, i.e.,
$N_D^-(v):=\{u \in V(D) \mid (u,v) \in A(D) \}$.

A hypergraph $\cH$ is \textit{interval}
if there exists a one-to-one function mapping the vertices of $V(\cH)$
to points on the real line such that for each hyperedge $e$,
there exists an interval containing the images of all elements of $e$,
but not the images of any vertices not in $e$.
There is a characterization
of interval hypergraphs by forbidden subhypergraphs:

\begin{Thm}[\cite{Moore}]\label{Thm:forbidden}
A hypergraph $\cH$ is an interval hypergraph if and only if
$\cH$ does not contain any of the hypergraphs in Figure~\ref{fig1}
as a subhypergraph.
\end{Thm}

\begin{figure}
\centering{
\psfrag{A}{$C_3$}
\psfrag{B}{$C_4$}
\psfrag{C}{$C_5$}
\psfrag{D}{$M_1$}
\psfrag{E}{$M_2$}
\psfrag{F}{$M_3$}
\psfrag{G}{$F_1$}
\psfrag{H}{$F_2$}
\psfrag{I}{$F_3$}
\psfrag{J}{$O_1$}
\psfrag{K}{$O_2$}
\includegraphics{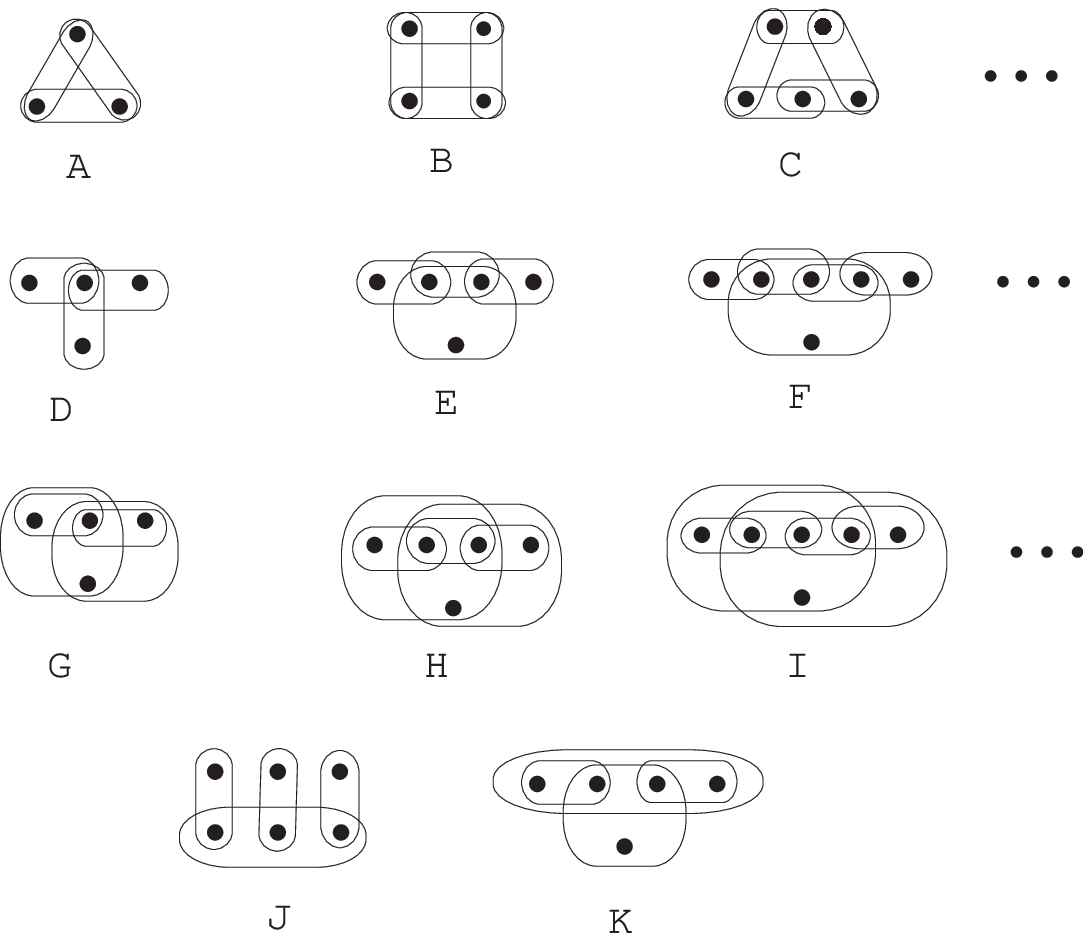}
\caption{Forbidden hypergraphs for interval hypergraphs}
\label{fig1}
}
\end{figure}

\noindent
More precisely, the hypergraphs in Figure~\ref{fig1} are defined as follows:
Given a positive integer $n \geq 3$, let
$C_n$ be the $2$-uniform hypergraph with $n$ vertices which forms a cycle,
and let $\mathcal{C} := \{ C_n \mid n \geq 3 \}$.
For a positive integer $n$,
we define hypergraphs $M_n$ and $F_n$ with $n+3$ vertices by
\begin{eqnarray*}
V(M_n)  &=& V(F_n) \ = \ \{v_1, v_2, \ldots, v_{n+3}\} \ = \ V, \\
E(M_n)  &=& \{  \{v_i,v_{i+1} \} \mid 1\le i \le n+1  \}
\cup \{ V\setminus\{v_1,v_{n+2}\} \}, \\
E(F_n) &=& \left\{  \, \{v_i,v_{i+1} \} \mid 1\le i \le n+1  \}
\cup \{ V\setminus\{v_1\},  V\setminus\{v_{n+2}\}  \,\right\}.
\end{eqnarray*}
Let $\mathcal{M} := \{ M_n \mid n \geq 1 \}$
and $\mathcal{F} := \{ F_n \mid n \geq 1 \}$.
Let $O_1$ be the hypergraph defined by
$V(O_1) = \{ x, x', y, y', z, z' \}$ and
$E(O_1) = \{ \{x, x' \}, \{y, y'\}, \{z, z'\},
\{x, y, z \} \}$,
and let
$O_2$ be the hypergraph defined by
$V(O_2) = \{ x, y, z, w, v \}$ and
$E(O_2) = \{ \{ x, y \}, \{ z, w\},
\{x, y, z, w\}, \{y, z, v \} \}$.
Theorem~\ref{Thm:forbidden} states that a hypergraph $\cH$ being
an interval hypergraph is equivalent to $\cH$ not containing
any of the hypergraphs in
$\mathcal{C} \cup \mathcal{M} \cup \mathcal{F} \cup \{O_1,O_2 \}$
as a subhypergraph.

First, we will show that the competition hypergraph of a DPO
may not be interval.
We will always embed the vertices of a DPO
(as well as the vertices of its competition hypergraph)
into $\mathbb{R}^2$ in a natural way.
For a positive integer $n$, we define
\begin{eqnarray*}
A_n &:=& \{
(i, n-i+1)
\in \mathbb{R}^2 \mid
i \in \{ 0, 1, 2, \ldots, n+1\} \}, \\
B_n &:=& \{
(i - \tfrac{1}{3}, n-i-\tfrac{1}{3}) \in \mathbb{R}^2 \mid
i \in \{ 0, 1, \ldots, n\}  \}.
\end{eqnarray*}
In the DPO defined on the set $A_n\cup B_n$,
two vertices $(i,n+1-i)$ and $(j,n+1-j)$ of $A_n$ with $i<j$
have a common out-neighbor $(i-\frac{1}{3},n-i-\frac{1}{3})$ if $j-i=1$
and have no common out-neighbor if $j-i \ge 2$.
Thus, the competition hypergraph of the DPO defined on the set $A_n\cup B_n$
is a path as a $2$-uniform hypergraph on the $n+2$ vertices in $A_n$
together with the $n+1$ isolated vertices in $B_n$.

\begin{Lem}\label{lem;M}
For a positive integer $n$,
there exists a doubly partial order whose competition hypergraph
contains $M_n$ as a subhypergraph.
\end{Lem}

\begin{proof}
For a positive integer $n$, we will define a DPO $D_n$
such that $C\cH(D_n)$ contains $M_n$ as a subhypergraph.
It is easy to check that, for the DPO $D_1$
defined on the set $A_1 \cup B_1 \cup
\{(0,0),(\frac{2}{3},\frac{2}{3})\}$
(see Figure~\ref{fig2}),
$C\cH(D_1)$ contains $M_1$ as a subhypergraph.
For a positive integer $n\ge 2$,
let $D_n$ be the DPO defined by $V(D_n)=A_n \cup B_n \cup \{(0,0)\}$
(see Figure~\ref{fig2_1}).
Then the hyperedges of $C\cH(D_n)$ consist of the hyperedges
of the $2$-uniform path induced by $A_n$
and the hyperedge
$N_{D_n}^-((0,0))= \left(A_n \setminus\{(0,n+1),(n+1,0)\} \right)
\cup \left(B_n\setminus  \left\{ \left(-\frac{1}{3},n-\frac{1}{3}\right),
\left(n-\frac{1}{3},-\frac{1}{3}\right) \right\} \right)$.
Note that
$(\frac{2}{3},n-\frac{4}{3}) \in B_n$ for $n\ge 2$.
Thus,
it is easy to see that the subhypergraph of $C\cH(D_n)$
induced by
$A_n \cup \left\{\left(\frac{2}{3}, n-\frac{4}{3}\right)\right\}$
is isomorphic to $M_n$.
\end{proof}

\begin{figure}[h]
\centering{
\psfrag{A}{$D_1$}
\psfrag{E}{$C\cH(D_1)$}
\includegraphics{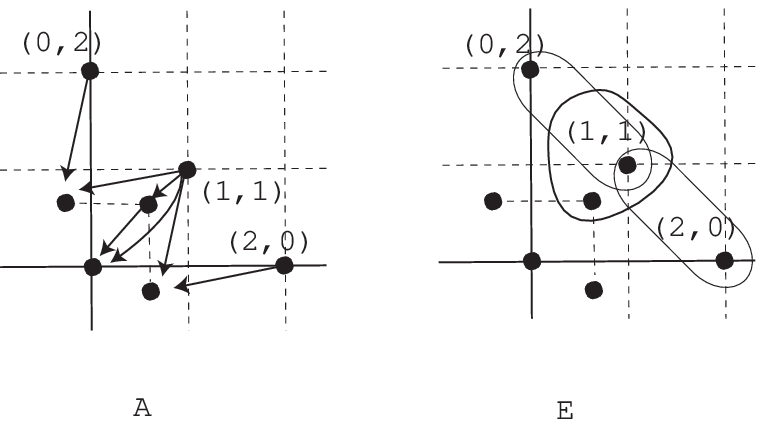}
\caption{The DPO $D_1$ and its competition hypergraph $C\cH(D_1)$}
\label{fig2}
}
\end{figure}

\begin{figure}[h]
\centering{
\psfrag{C}{$D_3$}
\psfrag{G}{$C\cH(D_3)$}
\includegraphics{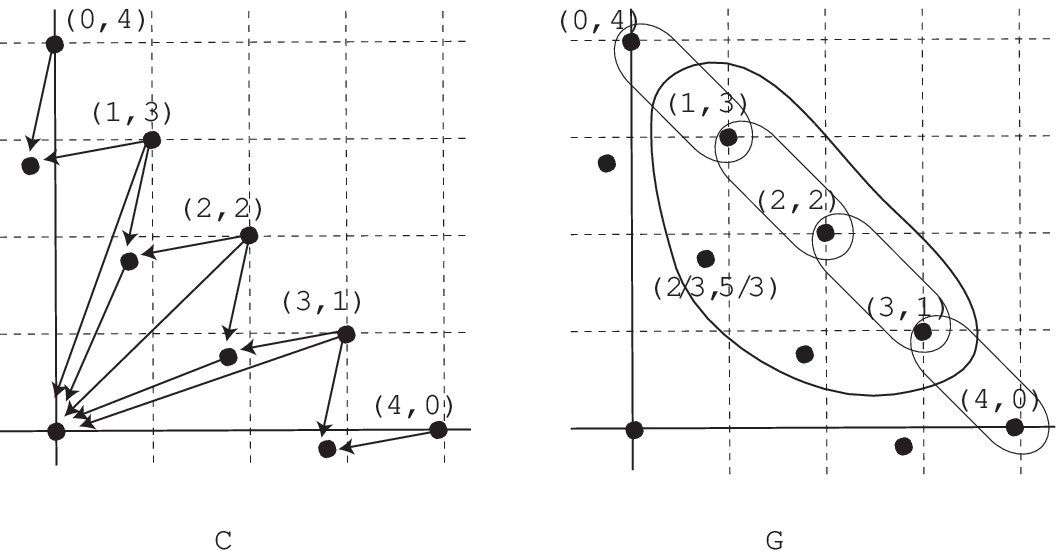}
\caption{The DPO $D_3$ and its competition hypergraph $C\cH(D_3)$}
\label{fig2_1}
}
\end{figure}

\begin{Lem}\label{lem;F}
For a positive integer $n$,
there exists a doubly partial order whose competition hypergraph
contains $F_n$ as a subhypergraph.
\end{Lem}

\begin{proof}
For a positive integer $n$, we will define a DPO $D'_n$
such that $C\cH(D'_n)$ contains $F_n$ as a subhypergraph.
It is easy to check that, for the DPO $D'_1$
defined
on the set $A_1 \cup B_1 \cup
\{(\frac{2}{3},\frac{2}{3}),(-1,0),(0,-1) \}$
(see Figure~\ref{fig4}), $C\cH(D'_1)$ contains $F_1$ as a subhypergraph.
For $n\ge 2$,
let $D'_n$ be the DPO $D'_n$ defined by
$V(D'_n)= A_n\cup B_n\cup \{ (0,-1), (-1,0) \}$
(see Figure~\ref{fig4_1}).
Then the hyperedges of $C\cH(D'_n)$ consist of the hyperedges
of the 2-uniform path induced by $A_n$ and
the hyperedges
$N_{D'_n}^-((-1,0)) = \left(A_n \setminus\{(n+1,0)\} \right)
\cup \left( B_n\setminus
\left\{ \left(n-\frac{1}{3},-\frac{1}{3}\right) \right\} \right)$
and $N_{D'_n}^-((0,-1))=\left(A_n \setminus\{(0,n+1)\} \right) \cup
\left( B_n \setminus \left\{ \left( -\frac{1}{3}, n - \frac{1}{3} \right)
\right\} \right)$.
Note that $( \frac{2}{3},n-\frac{4}{3} ) \in B_n$ for $n\ge 2$.
Thus,
it is easy to see that
the subhypergraph of $C\cH(D'_n)$ induced by
$A_n\cup \{ (\frac{2}{3},n-\frac{4}{3} ) \}$ is isomorphic to $F_n$.
\end{proof}

\begin{figure}[h]
\centering{
\psfrag{A}{$D'_1$}
\psfrag{B}{$C\cH(D'_1)$}
\includegraphics{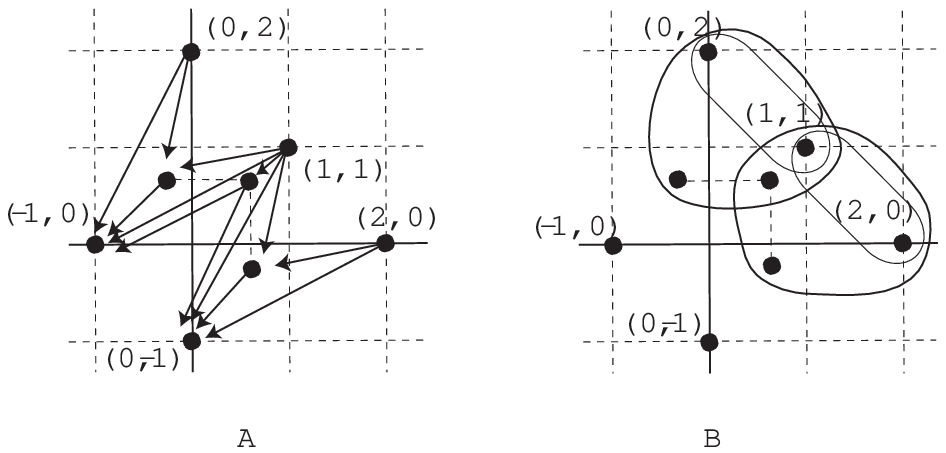}
\caption{The DPO $D'_1$ and its competition hypergraph $C\cH(D'_1)$}
\label{fig4}
}
\end{figure}

\begin{figure}[h]
\centering{
\psfrag{C}{$D'_3$}
\psfrag{D}{$C\cH(D'_3)$}
\includegraphics{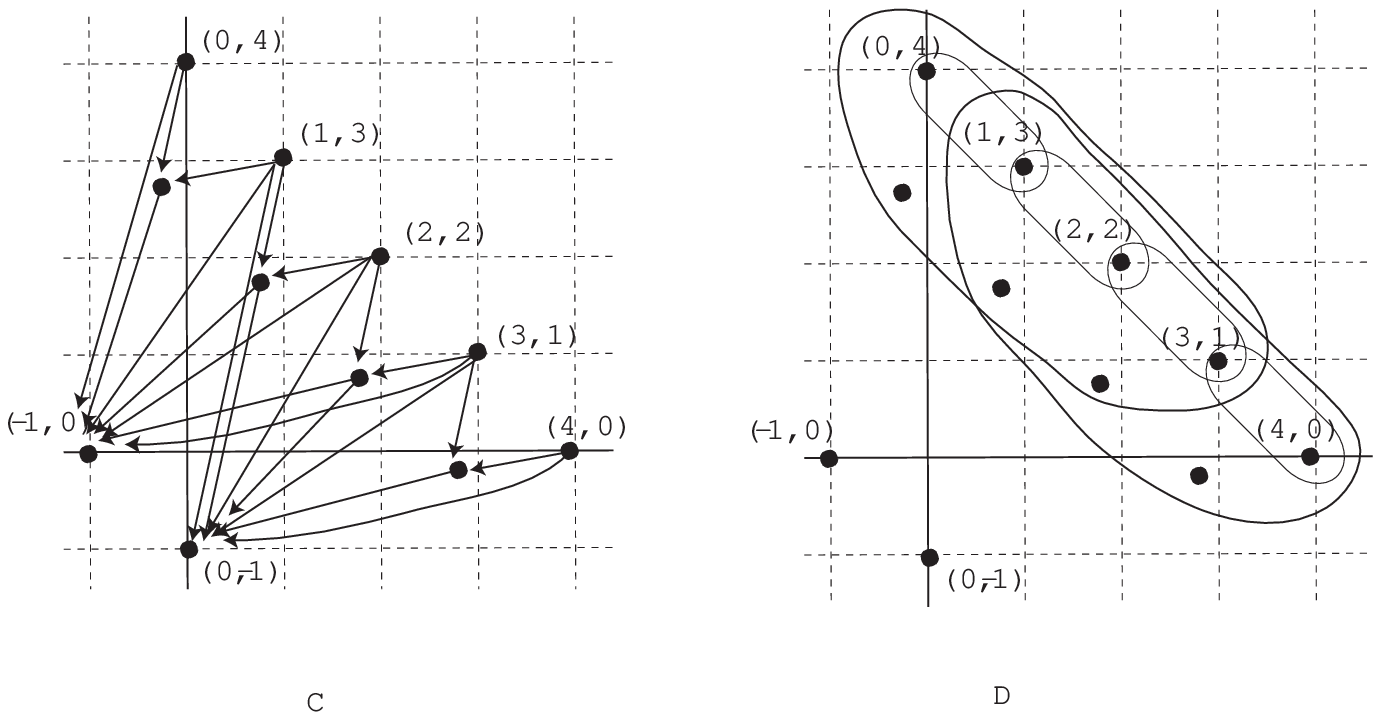}
\caption{The DPO $D'_3$ and its competition hypergraph $C\cH(D'_3)$}
\label{fig4_1}
}
\end{figure}

Lemmas~\ref{lem;M} and~\ref{lem;F} show that a hypergraph
isomorphic to an element in $\mathcal{M} \cup \mathcal{F}$
is realizable as a subhypergraph of the competition hypergraph of a DPO.

\begin{Thm}\label{main;interval}
For each hypergraph $\cH$ in $\mathcal{M} \cup \mathcal{F}$,
there exists a doubly partial order whose competition hypergraph
contains $\cH$ as a subhypergraph.
\end{Thm}

Now it is natural to ask whether the family $\mathcal{M} \cup \mathcal{F}$
contains all the forbidden subhypergraphs for
the competition hypergraph of a DPO being interval.
The answer is yes, as it will be shown 
in the remainder of this paper.

For the sake of simplicity, we define
an irreflexive and transitive relation $\searrow$
on $\mathbb{R}^2$ as follows:
For $x,y \in \mathbb{R}^2$ with
$x=(x_1,x_2)$ and $y=(y_1,y_2)$,
\[
x \searrow y \iff
x \neq y, \ x_1 \le y_1, \text{ and } x_2 \ge y_2.
\]
Obviously, for $x$, $y\in \mathbb{R}^2$ with $x\neq y$,
\begin{align*}
& x \not\prec y \mbox{ and } y \not\prec x
\mbox{ if and only if }
x \searrow y \mbox{ or } y \searrow x.
\tag{$\dagger$}
\end{align*}
The following lemmas are simple but useful:

\begin{Lem}\label{lem;prec}
For vertices $x$ and $y$ of a doubly partial order $D$,
if the competition hypergraph of $D$ has two hyperedges $e_x$ and $e_y$
such that $x\in e_x$, $y\not\in e_x$, $y\in e_y$, $x\not\in e_y$,
then either $x\searrow y$ or $y\searrow x$.
\end{Lem}

\begin{proof}
To reach a contradiction,
suppose that $x \not\searrow y$ and $y\not\searrow x$.
Then, by ($\dagger$), either $x \prec y$ or $y\prec x$.
If $x \prec y$, then any out-neighbor of $x$ is also an out-neighbor of $y$.
Therefore, any hyperedge containing $x$ contains $y$,
which is a contradiction to the existence of the hyperedge $e_x$.
Similarly, we can reach a contradiction if $y\prec x$.
Thus $x\searrow y$ or $y\searrow x$.
\end{proof}

\begin{Lem}\label{lem;prec2}
For vertices $x$, $y$ and $z$ of a doubly partial order $D$,
if $x \searrow y \searrow z$ then
$y$ and all of its in-neighbors are contained in any hyperedge
containing $x$ and $z$ in the competition hypergraph of $D$.
\end{Lem}

\begin{proof}
Let $e$ be a hyperedge containing $x$ and $z$ in $C\cH(D)$.
Then there exists a vertex $a$ such that $N_D^-(a)=e$.
Let $x=(x_1,x_2)$, $y=(y_1,y_2)$, $z=(z_1,z_2)$, and $a=(a_1,a_2)$.
Since $a\prec x$ and $a\prec z$,  $a_1 < \min\{x_1,z_1\}$
and $a_2 < \min\{x_2,z_2\}$.
Since  $x \searrow y \searrow z$, $x_1 \le y_1 \le z_1$
and $x_2 \ge y_2 \ge z_2$, which implies that
$\min\{x_1,z_1\} \le y_1$ and $\min\{x_2,z_2\} \le y_2$.
Therefore $a\prec y$ and thus $y\in e$.
Since $y\preceq u$ for any in-neighbor $u$ of $y$,
$a \prec u$
and hence $u\in e$.
\end{proof}

\begin{Lem}\label{lem;prec3}
For vertices $x$, $y$, $z$ of a doubly partial order $D$,
if $x \searrow y$, $z \prec x$, and $z \not\prec y$,
then $z \searrow y$.
\end{Lem}

\begin{proof}
Let $x=(x_1,x_2)$, $y=(y_1,y_2)$, and $z=(z_1,z_2)$.
Since $x \searrow y$, $x_1\le y_1$.
Since $z \prec x$, $z_1 \leq x_1$.
Therefore $z_1 \leq y_1$ and so $y \not\prec z$.
Since $y \not\prec z$ and $z\not\prec y$, either $z \searrow y$
or $y \searrow z$ by ($\dagger$).
If $y \searrow z$, then $y_1 \le z_1$ and so $x_1 \le z_1$,
a contradiction to the fact that $z \prec x$.
Thus $y \not\searrow z$ and so $z \searrow y$.
\end{proof}

\begin{Lem}\label{lem;prec4}
For vertices $x$, $y$, $z$ of a doubly partial order $D$,
if $x \searrow y$, $z \prec y$, and $z \not\prec x$,
then $x \searrow z$.
\end{Lem}

\begin{proof}
Let $x=(x_1,x_2)$, $y=(y_1,y_2)$, and $z=(z_1,z_2)$.
Since $x \searrow y$, $x_2 \ge y_2$.
Since $z \prec y$, $z_2 < y_2$.  Therefore $z_2 < x_2$.
Since $z \not\prec x$, $z_1 \ge x_1$ or $z_2  \ge  x_2$.
Since $z_2 < x_2$, $z_1 \ge x_1$.  Thus $x \searrow z$.
\end{proof}

Using Lemmas~\ref{lem;prec}, \ref{lem;prec2}, and \ref{lem;prec3},
we can show that the competition hypergraph of a DPO does not contain
any element of $\mathcal{C}\cup\{ O_1,O_2\}$ as  a subhypergraph.

\begin{Lem}\label{lem;C}
The competition hypergraph of a doubly partial order does not contain
any hypergraph in $\mathcal{C}$ as a subhypergraph.
\end{Lem}

\begin{proof}
We prove the lemma by contradiction.
Suppose that there exists a DPO $D$
such that $C\cH(D)$ contains $C_n$ as a subhypergraph for some $n\ge 3$.
Let $v_1v_2\cdots v_n$ be the vertices of $C_n$ such that $v_i$, $v_{i+1}$
are adjacent for all $1\le i \le n$, where the subscripts are
reduced modulo $n$.
Note that for any distinct $i$, $j$ in $\{1, 2, \ldots, n\}$,
there exists a hyperedge containing $v_{i}$ but not $v_j$.
Thus, by Lemma~\ref{lem;prec}, $v_i \searrow v_j$ or $v_j \searrow v_i$
for any distinct $i$, $j$ in $\{1, 2, \ldots, n\}$.
Without loss of generality, we may assume that $v_1\searrow v_2$.
If $v_3 \searrow v_1$, then $v_3 \searrow v_1 \searrow v_2$ and so,
by Lemma~\ref{lem;prec2}, $v_1$ and $v_3$ are adjacent, a contradiction.
Thus $v_1 \searrow v_3$.
Suppose that $v_3 \searrow v_2$.
Then $v_1 \searrow v_3 \searrow v_2$,
and by the argument in the proof of Lemma~\ref{lem;prec2},
a common out-neighbor of $v_1$ and $v_2$ is also an out-neighbor of $v_3$.
This implies that $C_n$ contains a hyperedge containing $\{v_1,v_2,v_3\}$,
which contradicts that $C_n$ is $2$-uniform.
Therefore, $v_2 \searrow v_3$.
By applying a similar argument,
we can claim that $v_3 \searrow v_4$.
By continuing this argument, we obtain
$v_1\searrow v_2 \searrow \cdots \searrow v_{n-1}
\searrow v_n \searrow v_1$.
Since the relation $\searrow$ is transitive,
we have $v_1 \searrow v_1$,
which is a contradiction to the irreflexivity of $\searrow$.
\end{proof}

\begin{Lem}\label{lem;O}
The competition hypergraph of a doubly partial order
does not contain $O_1$ or $O_2$ as a subhypergraph.
\end{Lem}

\begin{proof}
Suppose that there exists a DPO $D$ such that $C\cH(D)$ contains $O_1$
as a subhypergraph.
Let $\{ x,y,z\}$ be the hyperedge of $O_1$ of size $3$.
By Lemma~\ref{lem;prec}, any two $u$, $v$ of $\{x,y,z\}$ satisfy
$u\searrow v$ or $v\searrow u$. Thus $\alpha \searrow \beta\searrow \gamma$
for a permutation $(\alpha\beta\gamma)$ on $\{x,y,z\}$.
Without loss of generality, we may assume that $x \searrow y \searrow z$.
Let $y'$ be the vertex of $O_1$ other than $x,y,z$ that is adjacent to $y$.
By Lemma~\ref{lem;prec}, either $x \searrow y'$ or $y' \searrow x$,
and either $z \searrow y'$ or $y' \searrow z$.
If $y' \searrow x$, then $y' \searrow x \searrow y$
and so, by Lemma~\ref{lem;prec2}, $x$ and $y'$ are adjacent,
a contradiction. Thus $x \searrow y'$.
If $y' \searrow z$,  then $x \searrow y' \searrow z$
and so, by Lemma~\ref{lem;prec2}, $x$ and $y'$ are adjacent,
a contradiction.
Thus $z \searrow y'$. Then $y \searrow z \searrow y'$ and so,
by Lemma~\ref{lem;prec2}, $y'$ and $z$ are adjacent, a contradiction.
Hence, the competition hypergraph of a DPO
does not contain $O_1$ as a subhypergraph.

Suppose that there exists a DPO $D$ such that $C\cH(D)$
contains $O_2$ as a subhypergraph.
Let $V(O_2)=\{ x,y,z,w,v\}$ and let $\{ x,y,z,w\}$
be the hyperedge of $O_2$ of size $4$, and $\{x,y\}$ and $\{z,w\}$
be the hyperedges of size $2$. 
The vertices $y$ and $z$ satisfy the condition of Lemma~\ref{lem;prec}
and so $y \searrow z$ or $z \searrow y$.
Without loss of generality, we may assume that $y \searrow z$.
Now $x$ and $z$ satisfy the condition of Lemma~\ref{lem;prec}
and so $x \searrow z$ or $z \searrow x$.
If  $z \searrow x$ then $y \searrow z \searrow x$
and so $z$ belongs to any hyperedge containing $x$ and $y$
by Lemma~\ref{lem;prec2}, a contradiction.
Therefore $x \searrow z$.
Similarly we can show that $y \searrow w$.
Since $\{y,z,v\}, \{ x,y,z,w\} \in E(O_2)$,
there exist $a$, $b \in V(D)$ such that
$\{y,z,v\} = N_D^{-}(a) \cap V(O_2)$ and
$\{ x,y,z,w\}=N^{-}_D(b)\cap V(O_2)$.
If $a \prec w$, then $w\in N_D^{-}(a) \cap V(O_2)$, a contradiction.
If $w \prec a$, then $v\in N^-_D(b)\cap V(O_2)$, a contradiction.
Therefore $a \not\prec w$ and $w \not\prec a$.
Thus $y$, $w$, $a$ satisfy the condition of Lemma~\ref{lem;prec3}
and so $a \searrow w$.
If $a \prec x$ then $x\in N_D^{-}(a) \cap V(O_2)$, a contradiction.
Therefore $a \not\prec x$, and so
 $x$, $z$, $a$ satisfy the condition of Lemma~\ref{lem;prec4}
and so $x \searrow a$.
Thus $x \searrow a \searrow w$.
Since $N^{-}_D(b)$ is a hyperedge containing $x$ and $w$,
the vertex $v$ which is an in-neighbor of $a$ belongs to $N^{-}_D(b)$
by Lemma~\ref{lem;prec2}.
Since $v\in V(O_2)$, $v\in N^{-}_D(b)\cap V(O_2)$.
However, $\{ x,y,z,w\} = N^{-}_D(b) \cap V(O_2)$, a contradiction.
Hence the competition hypergraph of a DPO does not contain $O_2$
as a subhypergraph.
\end{proof}

By Lemmas~\ref{lem;C} and \ref{lem;O},
the following holds:

\begin{Thm}\label{thm;Main}
For a doubly partial order $D$,
the competition hypergraph of $D$ is interval
if and only if
it does not contain any hypergraph in $\mathcal{M}\cup\mathcal{F}$
as a subhypergraph.
\end{Thm}

It follows from Theorems~\ref{main;interval} and \ref{thm;Main}
that $\mathcal{M} \cup \mathcal{F}$ is the family
of all the forbidden subhypergraphs
for the competition hypergraph of a DPO being interval.
A hypergraph is \textit{chordal} if any cycle of length at least $4$
has two nonconsecutive vertices that are adjacent.
If a hypergraph does not contain any hypergraph in $\mathcal{C}$
as a subhypergraph, then it is chordal.
Therefore, by Lemma~\ref{lem;C}, the following corollary
immediately holds.

\begin{Cor}
The competition hypergraph of a doubly partial order is chordal.
\end{Cor}

Theorem~\ref{thm;DPOIntKim} shows that any interval graph
can be made into the competition graph of a DPO
by adding sufficiently many isolated vertices.
In the same context, we may ask whether
the following statement is true:
\begin{quote}
\begin{itemize}
\item[($\ast$)]
An interval hypergraph  can be made into
the competition hypergraph of a DPO
by adding sufficiently many isolated vertices.
\end{itemize}
\end{quote}
\begin{figure}
\centering{
\psfrag{a}{\small$v_1$}
\psfrag{b}{\small$v_2$}
\psfrag{c}{\small$v_3$}
\psfrag{d}{\small$v_4$}
\psfrag{e}{\small$v_5$}
\psfrag{f}{\small$v_6$}
\includegraphics{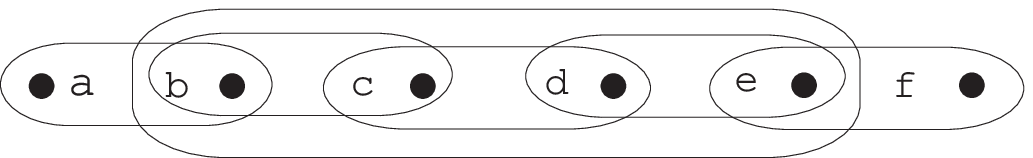}
\caption{An interval hypergraph $\cH$}\label{fig:inthyp}
}
\end{figure}
We can show that the answer is no by taking an interval hypergraph $\cH$
defined by
$V(\cH) = \{ v_1,v_2,v_3,v_4,v_5,v_6\}$ and
\[
E(\cH) = \{ \{v_1,v_2\} ,\{v_2,v_3\}, \{v_3,v_4\},
\{v_4,v_5\}, \{v_5,v_6\}, \{v_2,v_3,v_4,v_5\} \}
\]
(see Figure~\ref{fig:inthyp}).
To show by contradiction that $\cH$ cannot be made into
the competition hypergraph of a DPO even if
sufficiently many isolated vertices are added,
suppose that there is a DPO $D$ such that $C\cH(D)=\cH \cup I$
where $I$ is a set of isolated vertices.
One can observe that any two distinct vertices
among $v_2$, $v_3$, $v_4$, $v_5$ satisfy the condition
in Lemma~\ref{lem;prec}.
Thus, for any two distinct vertices $x$, $y$ in
the hyperedge  $\{v_2, v_3, v_4, v_5\}$,
either $x\searrow y$ or $y\searrow x$ holds.
By applying a similar argument for assuming that
$v_1 \searrow \cdots \searrow v_n$ in the proof of Lemma~\ref{lem;C},
we may assume without loss of generality that
$v_2 \searrow v_3 \searrow v_4 \searrow v_5$.
Since $\{v_3,v_4\}\in E(\cH)$,
there exists a vertex $u$
such that $N_D^{-}(u)=\{v_3,v_4\}$.
If $u=v_i$ for $i \in \{1,2,5,6\}$,
then $v_3$ and $v_4$ are contained in any hyperedge
containing $v_i$, a contradiction.
Thus $u \in I$.
Since $v_2 \searrow v_3$, $u \prec v_3$,
and $u \not\prec v_2$, it follows from Lemma~\ref{lem;prec4}
that $v_2 \searrow u$.
On the other hand, since $v_4 \searrow v_5$, $u \prec v_4$,
and $u \not\prec v_5$,
it follows from Lemma~\ref{lem;prec3} that $u \searrow v_5$.
We have shown that $v_2 \searrow u \searrow v_5$.
Then, by Lemma~\ref{lem;prec2}, $u$ is contained in
the hyperedge $\{v_2,v_3,v_4,v_5\}$,
which is a contradiction to $u \in I$.

\smallskip
The following statement weaker than $(\ast)$
seems to be worthy of mention:

\begin{Thm}\label{thm;Main2}
For an interval hypergraph $\cH$,
there exists a doubly partial order $D$
whose competition hypergraph contains $\cH$ as a subhypergraph.
\end{Thm}

\begin{proof}
Let $\cH$ be an interval hypergraph.
Then there exists an ordering $v_1,v_2,\ldots,v_n$ of the vertices
of $\cH$ such that any hyperedge of $\cH$ is consecutive in the ordering.
For each $i=1, 2, \ldots, n$, let $a_i := (i,n+1-i) \in \mathbb{R}^2$.
For each hyperedge $e$, we define 
$\min (e) := \min\{ j \mid v_j\in e\}$
and $\max (e):= \max \{ j \mid v_j\in e\}$
and let $b_e:=(\min (e)-1,n-\max (e)) \in \mathbb{R}^2$.
Let $D$ be a DPO on the set
$\{a_i \mid 1\le i\le n\} \cup \{  b_e \mid e\in E(\cH) \}$,
and
let $\cH'$ be the subhypergraph of $C\cH(D)$ induced by
the set $\{a_i \mid 1\le i\le n \}$.
We will show that the bijection
from $V(\cH)$ to $V(\cH')$
which maps a vertex $v_i$ to the vertex $a_i$
is an isomorphism between hypergraphs $\cH$ and $\cH'$.

Let $e$ be a hyperedge in $\cH$,
and let $i:=\min(e)$ and $k:=\max(e)$.
Then $e = \{v_i, v_{i+1}, \ldots, v_k \}$.
Since $e$ contains at least two vertices,
$i<k$.
Then $b_e=(i -1, n - k) \in \mathbb{R}^2$ is a vertex of $D$.
If $i \leq j \leq k$, then $i -1 < j$ and $n - k <  n + 1 - j$.
If $1 \leq j < i$, then $i - 1 \geq j$.
If $k < j \leq n$, then $n - k \geq n+1-j$.
Therefore, we have
$b_e=(i -1, n - k) \prec (j,n+1-j) = a_j$ if $i \leq j \leq k$
and
$b_e=(i -1, n - k) \not\prec (j,n+1-j) = a_j$ if
$1 \leq j < i$ or $k < j \leq n$.
Since $N_D^-(b_e)$ is a hyperedge of $C\cH(D)$,
$\{ a_i, a_{i+1}, \ldots, a_k \}$ is a hyperedge in $\cH'$.

Let $e'$ be a hyperedge in $\cH'$.
Let $i := \min \{ j \mid a_j \in e' \}$
and $k := \max \{ j \mid a_j \in e' \}$.
Note that $a_j \not\in e'$ for $1 \leq j < i$ and $k < j \leq n$.
Since $e'$ contains at least two vertices, $i < k$.
Then there exists a vertex $z = (z_1,z_2) \in V(D)$
such that $e'=N^-_D(z)$.
Note that $z \prec a_i$ and $z \prec a_k$.
If $i \leq j \leq k$, then $z = (z_1,z_2) \prec (j,n+1-j)=a_j$
since $z_1 < i$ and $z_2 < n+1-k$.
Therefore, $e' = \{a_i, a_{i+1}, \ldots, a_k\}$.
By the definition of $D$ and the fact that
$z \neq a_j$ for all $j \in \{1,2,\ldots,n\}$,
$z = b_{e} =(\min (e)-1,n-\max (e))$ for some hyperedge $e$ of $\cH$.
By the consecutive property of the ordering $v_1,v_2,\ldots,v_n$,
the hyperedge $e$ contains $v_j$ for each $j$
such that $\min (e) \le j \le \max (e)$.
Now we show that $\min (e)=i$ and $\max(e)=k$.
Since $\min (e)-1 = z_1 < i$ and $n-\max (e) = z_2 < n + 1 - k$,
we obtain $\min(e) \leq i$ and $k \leq \max(e)$.
Since $z \not\prec a_{i-1}$,
it holds that $\min(e)-1 \geq i-1$ or $n-\max(e) \geq n+1-(i-1)$.
If the latter happens, then $\max(e) \leq i-2$,
which contradicts the choice of $i$. Thus $\min(e) \geq i$.
Since $z \not\prec a_{k+1}$,
it holds that $\min(e)-1 \geq k+1$ or $n-\max(e) \geq n+1-(k+1)$.
By the choice of $k$, the latter holds and so $\max(e) \leq k$.
Thus $\min (e)=i$ and $\max(e)=k$.

Hence the theorem holds.
\end{proof}

\section*{Acknowledgments} 

The first author's work was supported by 
a National Research Foundation of Korea (NRF) grant 
funded by the Korean Government (MEST) (No. 2011-0005188). 
The third author's research was supported by 
a National Research Foundation of Korea Grant 
funded by the Korean Government, 
the Ministry of Education, Science and Technology (NRF-2011-357-C00004). 
The fourth author was supported by JSPS Research Fellowships 
for Young Scientists.


\end{document}